\documentclass[journal]{elsarticle}
\usepackage{epsfig}
\usepackage{amsmath,amsthm,extpfeil,subfigure}
\usepackage{amssymb}
\usepackage{graphicx,epstopdf,latexsym,amsfonts,fancyhdr}
\usepackage{booktabs,arydshln,multirow}
\usepackage{setspace}
\usepackage{xcolor}
\usepackage{stmaryrd}
\usepackage[utf8]{inputenc}
\usepackage{hyperref}
\usepackage{geometry}
\newtheorem{theorem}{Theorem}[section]

\newtheorem{lemma}[theorem]{Lemma}

\newtheorem{remark}[theorem]{Remark}

\journal{Journal of \LaTeX\ Templates}

\begin{document}

\begin{frontmatter}

\title{Analysis of discrete energy-decay preserving schemes for Maxwell's equations in Cole-Cole dispersive medium}

\author[mymainaddress,mysecondaddress]{Guoyu Zhang}
\ead{guoyu\_zhang@imu.edu.cn}

\author[thirdaddress]{Ziming Dong}
\ead{dong88math@163.com}

\author[mymainaddress,mysecondaddress]{Baoli Yin\corref{mycorrespondingauthor}}
\cortext[mycorrespondingauthor]{Corresponding author}
\ead{baolimath@126.com}

\author[mymainaddress,mysecondaddress]{Yang Liu}
\ead{mathliuyang@imu.edu.cn}

\author[mymainaddress,mysecondaddress]{Hong Li}
\ead{smslh@imu.edu.cn}

\address[mymainaddress]{School of Mathematical Sciences, Inner Mongolia University, Hohhot 010021, China;}
\address[mysecondaddress]{Inner Mongolia Key Laboratory of Mathematical Modeling and Scientific Computing, Hohhot 010021, China;}
\address[thirdaddress]{Faculty of Mathematics, Baotou Teachers’ College, Baotou 014030, China;}

\begin{abstract}
This work investigates the design and analysis of energy-decay preserving numerical schemes for Maxwell's equations in a Cole-Cole (C-C) dispersive medium. A continuous energy-decay law is first established for the C-C model through a modified energy functional. Subsequently, a novel \(\theta\)-scheme is proposed for temporal discretization, which is rigorously proven to preserve a discrete energy dissipation property under the condition \(\theta \in [\frac{\alpha}{2}, \frac{1}{2}]\). The temporal convergence rate of the scheme is shown to be first-order for \(\theta \neq 0.5\) and second-order for \(\theta = 0.5\). Extensive numerical experiments validate the theoretical findings, including convergence tests and energy-decay comparisons. The proposed SFTR-\(\theta\) scheme demonstrates superior performance in maintaining monotonic energy decay compared to an alternative 2nd-order fractional backward difference formula, particularly in long-time simulations, highlighting its robustness and physical fidelity.\end{abstract}


\end{frontmatter}

\section{Introduction}
The study of wave propagation in dispersive media such as water, soil, biological tissue, the ionosphere, plasma, optical fibers, and radar-absorbing materials has been an active area of engineering research since the early 1990s, motivated by the common characteristic of frequency-dependent permittivity or permeability in these materials \cite{biswas2017fractional,li2011developing,repo1996application,polk1995handbook, ionescu2017role, liu2023insight}. To characterize this frequency-domain behavior, various constitutive models have been developed, including the Drude model \cite{ziolkowski2001wave}, the Lorenz model \cite{smith2000negative}, and several anomalously dispersive models such as the Havriliak-Negami model \cite{havriliak1966complex,havriliak1967complex} and the Cole-Cole (C-C) model \cite{cole1941dispersion}.
Numerical simulation of wave propagation in such media has been addressed through multiple time-domain discretization techniques. These include finite-difference time-domain (FDTD) \cite{mustafa2014modeling, torres1996application,taflove2005computational,schuster1998fdtd,causley2011incorporating,rekanos2010auxiliary}, finite element time-domain (FETD) \cite{li2011developing,li2006analysis,jiao2001time,banks2009analysis}, spectral time-domain (STD) \cite{yang2021analysis,huang2019accurate}, and discontinuous Galerkin time-domain (DGTD) \cite{wang2021cg,lu2004discontinuous} methods. Complementary frequency-domain formulations have also been explored, as documented in \cite{mescia2014fractional,chakarothai2018novel,rekanos2012fdtd,rekanos2011fdtd,antonopoulos2017fdtd} and related references.
\par
The 2-D Maxwell's equations in a Cole-Cole dispersive medium \cite{li2011developing} can be stated as
\begin{eqnarray}\label{a1}
\epsilon_0 \epsilon_\infty \frac{\partial \boldsymbol E}{\partial t}(\boldsymbol x,t)&=&\nabla \times H(\boldsymbol x,t)-\frac{\partial \boldsymbol P}{\partial t}(\boldsymbol x,t),\\
\label{a2}
\mu_0 \frac{\partial H}{\partial t}(\boldsymbol x,t)&=&-\nabla \times \boldsymbol E(\boldsymbol x,t),\\
\label{a3}
\tau_0^\alpha \partial^\alpha_t \boldsymbol P(\boldsymbol x,t)+\boldsymbol P(\boldsymbol x,t)&=&\epsilon_0(\epsilon_s-\epsilon_\infty) \boldsymbol E(\boldsymbol x,t),
\end{eqnarray}
where $(\boldsymbol x,t)\in \Omega \times (0,T)$ for some $T>0$ and $\boldsymbol x=(x,y) \in \Omega=(a,b)\times (c,d)$.
The electric field \(\boldsymbol{E}(\boldsymbol{x},t) = (E_1, E_2)^\mathrm{T}\) and the magnetic field \(H(\boldsymbol{x},t)\) are the primary electromagnetic quantities in the model, while \(\boldsymbol{P}(\boldsymbol{x},t)\) denotes the time-domain polarization field. The material parameters include the free-space permittivity \(\epsilon_0\), the infinite-frequency permittivity \(\epsilon_\infty\), and the static permittivity \(\epsilon_s\), which satisfy the physical constraint \(\epsilon_s > \epsilon_\infty\). Additionally, \(\mu_0\) represents the free-space permeability, and \(\tau_0\) is the relaxation time.
The curl operators are defined as follows:
\[
\nabla \times H = \left( \frac{\partial H}{\partial y}, -\frac{\partial H}{\partial x} \right)^\mathrm{T}, \quad \nabla \times \boldsymbol{E} = \frac{\partial E_2}{\partial x} - \frac{\partial E_1}{\partial y}.
\]
To complete the initial-boundary value problem described by equations (\ref{a1})–(\ref{a3}), the following initial and boundary conditions are imposed:
\begin{equation}\begin{split}\label{a4}
\boldsymbol{E}(\boldsymbol{x},0) = \boldsymbol{E}_0(\boldsymbol{x}), \quad
H(\boldsymbol{x},0) = H_0(\boldsymbol{x}), \quad
\boldsymbol{P}(\boldsymbol{x},0) = \boldsymbol{0}, \quad \boldsymbol{x} \in \overline{\Omega},
\end{split}\end{equation}
and
\begin{equation}\begin{split}\label{a5}
\boldsymbol{n} \times \boldsymbol{E} = \boldsymbol{0} \quad \text{on } \partial\Omega \times [0,T],
\end{split}\end{equation}
where \(\boldsymbol{n}\) is the outward unit normal vector to the boundary \(\partial\Omega\).
The operator $\partial_t^\alpha$ represents the $\alpha$th fractional differential operator in Caputo sense, i.e.,
$\partial_t^\alpha u=\mathcal{I}^{1-\alpha} u_t$ where $\mathcal{I}^{\alpha}$ denotes the Riemann-Liouville fractional integral operator of order $\alpha$, defined by
\begin{equation}\label{a6}
(\mathcal{I}^\alpha u)(t)=\frac{1}{\Gamma(\alpha)}\int_{0}^{t}\frac{u(s)\mathrm{d}s}{(t-s)^{1-\alpha}}.
\end{equation}
\par
In their work \cite{li2011developing}, Li et al. established that the C-C model described by equations (\ref{a1})–(\ref{a5}) satisfies the following energy stability property:
\[
\mathcal{E}(t) \leq \mathcal{E}(0), \quad \text{for all } t \in [0,T],
\]
where the energy functional is defined as
\[
\mathcal{E}(t) = \epsilon_0(\epsilon_s - \epsilon_\infty)\left( \epsilon_0\epsilon_\infty\|\boldsymbol{E}(t)\|^2 + \mu_0\|H(t)\|^2 \right) + \|\boldsymbol{P}(t)\|^2.
\]
Subsequently, the authors of \cite{baoli2023discrete} provided, for the first time, a sufficient condition for determining when a second-order accurate discrete scheme can preserve such an energy stability property.
See also the related work \cite{xiao2025unconditionally}.
Nevertheless, it remains an important objective to define an appropriate energy functional that exhibits monotonic decay over time, and to develop numerical schemes that preserve this energy-decay behavior. Such an energy functional is particularly valuable in the design of adaptive time-stepping strategies, as the rate of energy-decay can serve as an effective criterion for dynamically adjusting the time step size.
In light of this, the contributions of this paper can be summarized as follows:
\begin{itemize}
\item A rigorous proof of the energy-decay law for the Cole-Cole model is established.
\item A $\theta$-scheme is designed and theoretically demonstrated to preserve the discrete energy-decay property.
\item Comprehensive numerical experiments are provided to validate the theoretical findings.
\end{itemize}
\par
This paper is organized as follows. 
Section 2 establishes the energy-decay law for the Cole-Cole model through the introduction of a modified energy functional. 
In Section 3, we propose a $\theta$-scheme and demonstrate its capability to preserve the discrete energy-decay property by proving several essential characteristics of the method. 
A comprehensive error analysis for the temporally semi-discrete scheme is presented in Section 4, while extensive numerical validation is provided in Section 5. 
The paper concludes with final remarks in Section 6.
\par
This section concludes with essential notation and assumptions. Let \( C \) denote a generic positive constant, independent of the mesh sizes \( h \) and \( \tau \), whose value may vary in different contexts.
For \( r \geq 0 \), \( H^r(\Omega) \) denotes the standard Sobolev space with norm \( \|\cdot\|_r \) and seminorm \( |\cdot|_r \). Here, \( \|\cdot\| \) abbreviates \( \|\cdot\|_0 \), and \( H^0(\Omega) \) coincides with \( L^2(\Omega) \). Define the space
\[
H^r(\mathrm{curl};\Omega) = \left\{ \boldsymbol{v} \in (H^r(\Omega))^2 : \nabla \times \boldsymbol{v} \in H^r(\Omega) \right\},
\quad
\|\boldsymbol{v}\|_{r,\mathrm{curl}} = \left( \|\boldsymbol{v}\|_r^2 + \|\nabla \times \boldsymbol{v}\|_r^2 \right)^{1/2},
\]
and its subspace
\[
H_0(\mathrm{curl};\Omega) = \left\{ \boldsymbol{v} \in H(\mathrm{curl};\Omega) : \boldsymbol{n} \times \boldsymbol{v} = 0 \text{ on } \partial\Omega \right\},
\]
where \( H(\mathrm{curl};\Omega) = H^0(\mathrm{curl};\Omega) \).
We assume the solution \( \boldsymbol{P} \) satisfies the regularity condition: for almost every \( \boldsymbol{x} \in \Omega \), \( \partial_t \boldsymbol{P}(\boldsymbol{x}, t) \in L^1(0,T) \) and \( \partial_t^\alpha \boldsymbol{P}(\boldsymbol{x}, t) \in C[0,T] \). Under this assumption, the following identities (see \cite[Theorem 2.14]{diethelm2010analysis}) and inequality (see \cite{alsaedi2015maximum,liao2021energy}) hold:
\begin{equation}\label{a7}\begin{split}
\partial_t \boldsymbol{P} = {}_{RL}\partial_t^{1-\alpha} \left( \mathcal{I}^{1-\alpha} \partial_t \boldsymbol{P} \right) = {}_{RL}\partial_t^{1-\alpha} \left( \partial_t^\alpha \boldsymbol{P} \right),
\end{split}
\end{equation}
and
\begin{equation}\label{a8}\begin{split}
\partial_t^\alpha \boldsymbol{P}(t) \cdot \left[ {}_{RL}\partial_t^{1-\alpha} \left( \partial_t^\alpha \boldsymbol{P} \right) \right](t) \geq \frac{1}{2} \left[ {}_{RL}\partial_t^{1-\alpha} \left( \partial_t^\alpha \boldsymbol{P} \right)^2 \right](t) + \frac{1}{2} \eta_\alpha(t) \left( \partial_t^\alpha \boldsymbol{P}(t) \right)^2,
\end{split}
\end{equation}
where \( \eta_\alpha(t) = \frac{t^{\alpha-1}}{\Gamma(\alpha)} \), and \( {}_{RL}\partial_t^{\alpha} u = \frac{\mathrm{d}}{\mathrm{d}t} \mathcal{I}^{1-\alpha} u \) denotes the Riemann–Liouville fractional derivative of order \( \alpha \).

\section{Continuous energy dissipation}
The weak formulation of (\ref{a1})-(\ref{a3}) can be formulated as:
Find $\boldsymbol E \in C(0,T; H_0({\rm curl};\Omega))\cap C^1(0,T;(L^2(\Omega))^2)$,
$H\in C^1(0,T;L^2(\Omega))$ and $\boldsymbol P \in C^1(0,T;(L^2(\Omega))^2)$ so that
\begin{eqnarray}\label{b1}
\epsilon_0 \epsilon_\infty \bigg(\frac{\partial \boldsymbol E}{\partial t},\boldsymbol\chi\bigg)
+\bigg(\frac{\partial \boldsymbol P}{\partial t},\boldsymbol\chi\bigg)
-(H,\nabla \times \boldsymbol\chi)&=&0,
\quad \forall \boldsymbol\chi \in H_0({\rm curl};\Omega),
\\
\label{b2}
\mu_0 \bigg(\frac{\partial H}{\partial t},\phi\bigg)
+(\nabla \times \boldsymbol E,\phi)&=&0,
\quad \forall \phi \in L^2(\Omega),\\
\label{b3}
\tau_0^\alpha \big(\partial^\alpha_t \boldsymbol P,\boldsymbol\psi\big)+\big(\boldsymbol P,\boldsymbol\psi\big)
-\epsilon_0(\epsilon_s-\epsilon_\infty) \big(\boldsymbol E,\boldsymbol\psi\big)&=&0,
\quad \forall \boldsymbol\psi \in (L^2(\Omega))^2.
\end{eqnarray}
For a better explanation of the definition of compatible energy, we shall present first the energy dissipation law of the classical Debye model (i.e., by setting $\alpha=1$ in (\ref{a3}) or (\ref{b3})).
\begin{lemma}(Energy dissipation of Debye model of order 1)\label{lem.1}
Let $\alpha=1$ in (\ref{b3}).
There holds
\begin{equation}\label{b4}\begin{split}
\frac{\mathrm{d}}{\mathrm{d}t}\mathcal{E}(t)+\tau_0\|\partial_t \boldsymbol P\|^2 \leq 0,
\end{split}
\end{equation}
where $\mathcal{E}(t):=\tau_0\int_{0}^{t}\|\partial_t \boldsymbol P(s)\|^2\mathrm{d}s
+\|\boldsymbol P\|^2
+\epsilon_0(\epsilon_s-\epsilon_\infty)\big(\epsilon_0\epsilon_\infty\|\boldsymbol E\|^2+\mu_0\|H\|^2\big)$ and therefore,
\begin{equation}\label{b4.0}
\mathcal{E}(t_1)\leq \mathcal{E}(t_2),\quad\text{for } t_1\geq t_2.
\end{equation}
\end{lemma}
\begin{proof}
Let $\boldsymbol \chi=\boldsymbol E$ in (\ref{b1}) and $\phi=H$ in (\ref{b2}), and add these two resultants together to obtain
\begin{equation}\label{b4.1}\begin{split}
\frac{1}{2}\epsilon_0\epsilon_\infty\frac{\mathrm{d}}{\mathrm{d}t}\|\boldsymbol E\|^2
+\frac{1}{2}\mu_0\frac{\mathrm{d}}{\mathrm{d}t}\|H\|^2
+\bigg(\frac{\partial \boldsymbol P}{\partial t}, \boldsymbol E\bigg)=0.
\end{split}
\end{equation}
Multiplying both sides of (\ref{b4.1}) by $\epsilon_0(\epsilon_s-\epsilon_\infty)$, and adding the resultant to the equation (\ref{b3}) after replacing $\boldsymbol \psi$ with $\frac{\partial \boldsymbol P}{\partial t}$, one can readily arrive at (\ref{b4}).
Then, (\ref{b4.0}) follows for $\frac{\mathrm{d}}{\mathrm{d}t}\mathcal{E}(t) \leq 0$ by (\ref{b4}), which completes the proof of the lemma.
\end{proof}
\begin{lemma}(Energy dissipation of Cole-Cole model)\label{lem.2}
For arbitrary $\alpha \in (0,1)$, there holds
\begin{equation}\label{b5}\begin{split}
\frac{\rm d}{{\rm d}t}\mathcal{E}_\alpha(t)+\tau_0^\alpha \eta_\alpha(t)\|\partial_t^\alpha \boldsymbol P\|^2\leq 0,
\end{split}
\end{equation}
where $\mathcal{E}_\alpha(t):=\tau_0^\alpha \mathcal{I}^\alpha\|\partial_t^\alpha \boldsymbol P\|^2
+
\|\boldsymbol P\|^2
+\epsilon_0(\epsilon_s-\epsilon_\infty)\big(\epsilon_0\epsilon_\infty\|\boldsymbol E\|^2+\mu_0\|H\|^2\big)$ and therefore,
\begin{equation}\label{b5.1}
\mathcal{E}_\alpha(t_1)\leq \mathcal{E}_\alpha(t_2),\quad\text{for }t_1\geq t_2.
\end{equation}
\end{lemma}
\begin{proof}
Replacing $\boldsymbol \psi$ by $\frac{\partial \boldsymbol P}{\partial t}$ in (\ref{b3}), multiplying both sides of (\ref{b4.1}) by $\epsilon_0(\epsilon_s-\epsilon_\infty)$, and adding these two resultants together, one gets
\begin{equation}\label{b6}
\tau_0^\alpha \big(\partial^\alpha_t \boldsymbol P,\partial_t \boldsymbol P\big)
+\frac{1}{2}\frac{\rm d}{{\rm d}t}\|\boldsymbol P\|^2
+\frac{1}{2}\epsilon_0^2\epsilon_\infty(\epsilon_s-\epsilon_\infty)\frac{\rm d}{{\rm d}t}\|\boldsymbol E\|^2
+\frac{1}{2}\mu_0\epsilon_0(\epsilon_s-\epsilon_\infty)\frac{\rm d}{{\rm d}t}\|H\|^2
=0.
\end{equation}
Then, (\ref{b5}) follows in accordance to (\ref{a7}), (\ref{a8}) and (\ref{b6}), and (\ref{b5.1}) is due to the monotonicity of the function $\mathcal{E}_\alpha(t)$ by (\ref{b5}).
\end{proof}
\begin{remark}\label{rem.1}
Clearly, the term $\mathcal{I}^\alpha u$ tends to $\int_{0}^{t}u(s)\mathrm{d}s$ and $\eta_\alpha(t)=\frac{t^{\alpha-1}}{\Gamma(\alpha)} \to 1$ when $\alpha \to 1$, and therefore, (\ref{b5}) extends (\ref{b4}) to the fractional case naturally.
The energy $\mathcal{E}_\alpha(t)$ is thus regarded as being compatible with the original one $\mathcal{E}(t)$.
\end{remark}
\section{Discrete energy dissipation}
In temporal direction we approximate the Maxwell's equations based on a uniform mesh with grid points:
$0=t_0<t_1<\cdots<t_N=T$ with $t_n=n\tau$ and $\tau=\frac{T}{N}$.
Let $u^n:=u(t_n)$ for brevity. Introduce the following symbols
\begin{equation}\label{c1}\begin{split}
\partial_\tau^{n-\theta} u:=\frac{u^n-u^{n-1}}{\tau},\quad
\overline{u}^{n-\theta}:=(1-\theta)u^n+\theta u^{n-1}.
\end{split}
\end{equation}
as approximation to $u_t$ and $u$ at time $t_{n-\theta}$, respectively.
To approximate the term $\partial^\alpha_t u(t_{n-\theta})$, we adopt the shifted fractional trapezoidal rule \cite{yin2020necessity,yin2021efficient}, i.e., let
\begin{equation}\label{c2}\begin{split}
\partial_\tau^\alpha u^{n-\theta}:=\tau^{-\alpha}\sum_{k=1}^{n}\omega^{n-k}(u^k-u^0),
\end{split}
\end{equation}
where the weights $\omega_k$'s, which generally depend on $\alpha$ and $\theta$, are generated by the function $\omega(\zeta)$:
\begin{equation}\label{c3}
\omega(\zeta)=\bigg[\frac{1-\zeta}{\frac{1}{2}(1+\zeta)+\frac{\theta}{\alpha}(1-\zeta)}\bigg]^\alpha,\quad 0<\theta\leq \frac{1}{2}.
\end{equation}
We also introduce a new sequence $\{\varpi_k\}_{k=0}^\infty$ aiming to facilitate the following analysis, i.e., define
\begin{equation}\label{c3.0}
\varpi(\zeta)=\sum_{k=0}^{\infty}\varpi_k \zeta^k:=\frac{1-\zeta}{\omega(\zeta)}
=(1-\zeta)^{1-\alpha}\bigg[\frac{1}{2}(1+\zeta)+\frac{\theta}{\alpha}(1-\zeta)\bigg]^\alpha.
\end{equation}
\begin{lemma}\label{lem.3}
For $\alpha \in (0,1)$ and $\theta \in (0,\frac{1}{2}]$, the sequence $\{\varpi_k\}_{k=0}^\infty$ fulfills the properties that {\rm(i)} $\varpi_k=O(k^{\alpha-2})$ and {\rm(ii)} $\tau^{-(1-\alpha)}e^{(\frac{1}{2}-\theta)\tau}\varpi(e^{-\tau})=1+O(\tau^2)$.
\end{lemma}
\begin{proof}
Let $\nu(\zeta):=\big[\frac{1}{2}(1+\zeta)+\frac{\theta}{\alpha}(1-\zeta)\big]^\alpha$ and thus $\varpi(\zeta)=(1-\zeta)^{1-\alpha}\nu(\zeta)$.
The only singular point $\zeta_0$ of $\nu(\zeta)$ is
\[
\zeta_0=1+\frac{1}{\frac{\theta}{\alpha}-\frac{1}{2}},
\]
satisfying that $|\zeta_0|>1$ due to the fact $\frac{\theta}{\alpha}>0$.
Then, $\nu(\zeta)$ is analytic on a closed unit disk in $\mathbb{C}$, indicating that the weights of the expansion of $\nu(\zeta)$ decay faster than $O(k^{-\ell})$ for any positive integer $\ell$.
The decay rate of the weights $\varpi_k$, as a discrete convolution of two sequences formed by the coefficients of expansion of $(1-\zeta)^{1-\alpha}$ and $\nu(\zeta)$ respectively, is totally determined by $(1-\zeta)^{1-\alpha}$, which confirms (i).
To prove (ii), replace $e^{-\tau}$ and $e^{(\frac{1}{2}-\theta)\tau}$ with their Taylor expansions with respect to $\tau$, followed by some simple calculations which are omitted here.
\end{proof}
\begin{remark}\label{rem.2}
The properties (i) and (ii) in the above lemma essentially imply the formula $\tau^{\alpha-1}\sum_{k=0}^{n}\varpi_k u^{n-k}$ can approximate $\partial_t^{1-\alpha}u$ at time $t_{n-\frac{1}{2}+\theta}$ with second-order accuracy.
We refer to the shifted convolution quadrature in \cite{liu2019unified} for more information.
\end{remark}
\begin{lemma}\label{lem.4}
For $\theta \in [\frac{\alpha}{2},\frac{1}{2}]$, the weights $\varpi_k$ satisfy {\rm(i)} $\varpi_0>0$, {\rm(ii)} $\varpi_k\leq 0$, $k\geq 1$.
\end{lemma}
\begin{proof}
We present first a recursive formula for calculating $\varpi_k$.
Let $\zeta=0$ in (\ref{c3.0}) and thus $\varpi_0=\varpi(0)=(\frac{1}{2}+\frac{\theta}{\alpha})^\alpha>0$ which confirms (i).
Take the first derivative of $\varpi(\zeta)$ with respect to $\zeta$ and multiply both sides of the resultant by $(1-\zeta)\big[\frac{1}{2}(1+\zeta)+\frac{\theta}{\alpha}(1-\zeta)\big]$, to obtain
\begin{equation}\label{c3.0.1}
(1-\zeta)\bigg[\frac{1}{2}(1+\zeta)+\frac{\theta}{\alpha}(1-\zeta)\bigg]\varpi'(\zeta)
=
\bigg[\bigg(\frac{\theta}{\alpha}-\frac{1}{2}\bigg)\zeta+\alpha-\frac{\theta}{\alpha}-\frac{1}{2}\bigg]\varpi(\zeta).
\end{equation}
Considering that $\varpi(\zeta)=\sum_{k=0}^{\infty}\varpi_k\zeta^k$ and $\varpi'(\zeta)=\sum_{k=1}^{\infty}k\varpi_k\zeta^{k-1}$, by comparing the coefficient of $\zeta^k$ of the expansion of both sides of (\ref{c3.0.1}), we have
\begin{equation}\label{c3.0.2}
\varpi_1=\frac{\alpha-\frac{\theta}{\alpha}-\frac{1}{2}}{\frac{\theta}{\alpha}+\frac{1}{2}}\varpi_0
\leq \frac{\alpha-1}{\frac{\theta}{\alpha}+\frac{1}{2}}\varpi_0<0,\quad\text{since $\theta\geq \frac{\alpha}{2}$},
\end{equation}
and for $k\geq 2$,
\begin{equation}\label{c3.0.3}
\varpi_k=\frac{1}{\big(\frac{\theta}{\alpha}+\frac{1}{2}\big)k}
\bigg\{\bigg[\frac{\theta}{\alpha}(2k-3)+\alpha-\frac{1}{2}\bigg]\varpi_{k-1}
+\bigg(\frac{\theta}{\alpha}-\frac{1}{2}\bigg)(3-k)\varpi_{k-2}\bigg\}.
\end{equation}
\par
Next we prove that if $k\geq 4$, there holds
\begin{equation}\label{c3.0.4}
\frac{\varpi_{k-1}}{\varpi_{k-2}}\geq \rho_{k-1}
>\frac{(\frac{\theta}{\alpha}-\frac{1}{2})(k-3)}{\frac{\theta}{\alpha}(2k-3)+\alpha-\frac{1}{2}},
\end{equation}
where $\rho(\alpha,\theta,k)=\frac{(\frac{\theta}{\alpha}-\frac{1}{2})(k-2)}{\frac{\theta}{\alpha}(2k-1)+\alpha-\frac{1}{2}}\frac{4\theta}{2\theta+\alpha}(1+\frac{1}{k})$.
Note that the correctness of the second inequality of (\ref{c3.0.4}) is obvious.
By letting $k=3$ in (\ref{c3.0.3}), we have $\frac{\varpi_3}{\varpi_2}=(\frac{3\theta}{\alpha}+\alpha-\frac{1}{2})/(\frac{3\theta}{\alpha}+\frac{3}{2})$ and
\begin{equation}\label{c3.0.5}\begin{split}
\frac{\varpi_3}{\varpi_2}-\rho_3 &=
\frac{4\alpha^4-4\alpha^3+(32\theta+1)\alpha^2+28\theta^2}{3(\alpha+2\theta)(2\alpha^2-\alpha+10\theta)}
\\&\geq
\frac{4\alpha^4+12\alpha^3+8\alpha^2}{3(\alpha+2\theta)(2\alpha^2-\alpha+10\theta)}>0.
\end{split}\end{equation}
Now, assume (\ref{c3.0.4}) is true for $k=4,5,\cdots,m (m\geq 4)$.
For $k=m+1$, we derive from (\ref{c3.0.3}) that
\begin{equation}\label{c3.0.6}
\frac{\varpi_m}{\varpi_{m-1}}=\frac{1}{\big(\frac{\theta}{\alpha}+\frac{1}{2}\big)m}
\bigg[\frac{\theta}{\alpha}(2m-3)+\alpha-\frac{1}{2}
+\bigg(\frac{\theta}{\alpha}-\frac{1}{2}\bigg)(3-m)\frac{\varpi_{m-2}}{\varpi_{m-1}}\bigg],
\end{equation}
which, combined with the assumption, leads to
\begin{equation}\label{c3.0.7}\begin{split}
\frac{\varpi_m}{\varpi_{m-1}}&\geq\frac{1}{\big(\frac{\theta}{\alpha}+\frac{1}{2}\big)m}
\bigg[\frac{\theta}{\alpha}(2m-3)+\alpha-\frac{1}{2}
+\bigg(\frac{\theta}{\alpha}-\frac{1}{2}\bigg)(3-m)\frac{1}{\rho_{m-1}}\bigg]
\\&=
\frac{1}{\big(\frac{\theta}{\alpha}+\frac{1}{2}\big)m}
\bigg[\frac{\theta}{\alpha}(2m-3)+\alpha-\frac{1}{2}\bigg]
\bigg(1+\frac{2\theta+\alpha}{4\theta}\frac{m-1}{m}\bigg)
\\&=:\widetilde{\rho}(\alpha,\theta,m).
\end{split}\end{equation}
To avoid tedious calculation, we shall numerically show that $\widetilde{\rho}(\alpha,\theta,m) \geq \rho(\alpha,\theta,m)$.
Let $m=\frac{4}{x}$ for some $x\in(0,1]$, and define a function $\Theta(x,\alpha,\theta)$ in
$\Sigma=\{(x,\alpha,\theta):x\in (0,1],\alpha\in (0,1),\theta \in[\frac{\alpha}{2},\frac{1}{2}]\}$,
such that $\Theta(x,\alpha,\theta)=\widetilde{\rho}(\alpha,\theta,\frac{4}{x})-\rho(\alpha,\theta,\frac{4}{x})$.
\begin{figure}[htbp]
\centering
\subfigure[]{
\begin{minipage}[t]{0.6\linewidth}
\centering
\includegraphics[width=1\textwidth]{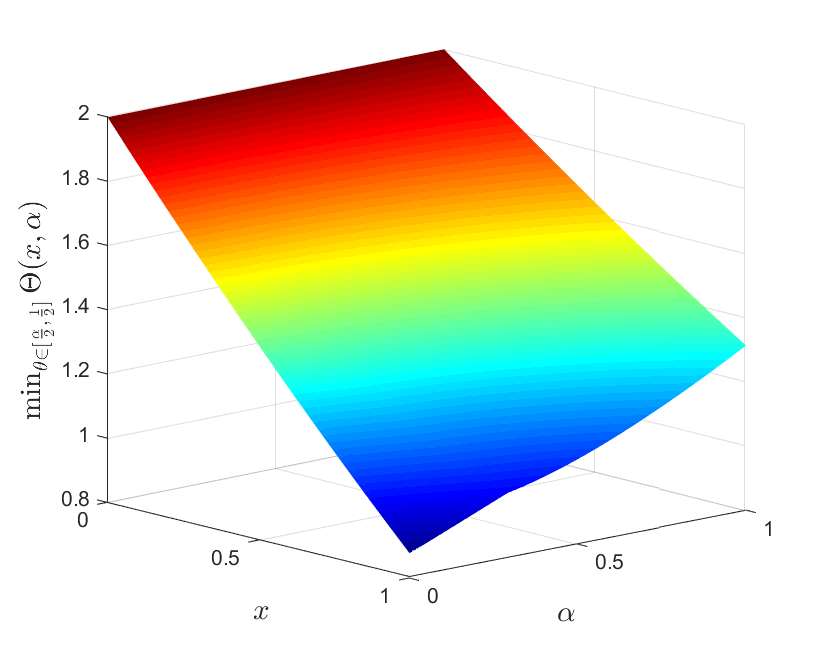}
\end{minipage}%
}%
\centering
\caption{Values of $\min_{\theta \in [\frac{\alpha}{2},\frac{1}{2}]}\Theta(x,\alpha,\theta)$ for $(x,\alpha)\in (0,1]\times (0,1)$.}\label{C1}
\end{figure}
Fig.\ref{C1} illustrates the values $\min_{\theta \in [\frac{\alpha}{2},\frac{1}{2}]}\Theta(x,\alpha,\theta)$ for $(x,\alpha)\in (0,1]\times (0,1)$, which indicates the positivity of $\min_{\theta \in [\frac{\alpha}{2},\frac{1}{2}]}\Theta(x,\alpha,\theta)$ and further $\widetilde{\rho}(\alpha,\theta,m)\geq \rho(\alpha,\theta,m)$.
Therefore, combining (\ref{c3.0.5})-(\ref{c3.0.7}) we justify (\ref{c3.0.4}) by induction.
\par
Finally, by (\ref{c3.0.3}) and (\ref{c3.0.4}), we derive for $k\geq 4$,
\begin{equation}\label{c3.0.8}
\varpi_k=\frac{\varpi_{k-1}}{\big(\frac{\theta}{\alpha}+\frac{1}{2}\big)k}
\bigg\{\bigg[\frac{\theta}{\alpha}(2k-3)+\alpha-\frac{1}{2}\bigg]
+\bigg(\frac{\theta}{\alpha}-\frac{1}{2}\bigg)(3-k)\frac{\varpi_{k-2}}{\varpi_{k-1}}\bigg\}
<0,
\end{equation}
provided $\varpi_{k-1}<0$.
Direct calculations from (\ref{c3.0.3}) show that
\begin{equation}\label{c3.0.9}
\varpi_2=\frac{\alpha(\alpha-1)}{2\big(\frac{\theta}{\alpha}+\frac{1}{2}\big)^2}\varpi_0<0,\quad \text{and}\quad
\varpi_3=\frac{1}{3\big(\frac{\theta}{\alpha}+\frac{1}{2}\big)}\bigg(\frac{3\theta}{\alpha}+\alpha-\frac{1}{2}\bigg)\varpi_2<0.
\end{equation}
Then, the results (\ref{c3.0.2}), (\ref{c3.0.8}) and (\ref{c3.0.9}) validate (ii), which complete the proof of the lemma.
\end{proof}
\begin{remark}\label{rem.3}
We emphasize that the condition $\theta \in [\frac{\alpha}{2},\frac{1}{2}]$ in the above lemma is not necessary to guarantee (i) and (ii), however, the negativity of $\varpi_1$ in (\ref{c3.0.2}) actually requires $\theta>\alpha(\alpha-\frac{1}{2})$, meaning that $\theta \geq \frac{1}{2}$ (and therefore $\theta=\frac{1}{2}$ due to (\ref{c3})) when $\alpha \to 1$, which is still compatible to our condition $\theta \in [\frac{\alpha}{2},\frac{1}{2}]$.
\end{remark}
\begin{lemma}\label{lem.5}
Let $\theta \in [\frac{\alpha}{2},\frac{1}{2}]$.
Assume $v_0=0$.
For any sequence $(v_1,v_2,\cdots,v_n)\in \mathbb{R}^n, n\geq 1$, there holds
\begin{equation}\label{c3.1}
 v^n\sum_{k=1}^{n}\varpi_{n-k}v^k
 \geq \frac{1}{2}\sum_{k=1}^{n}a_{n-k}\big[(v^k)^2-(v^{k-1})^2\big]
 +\frac{1}{2\varpi_0}\bigg\|\sum_{k=1}^{n}\varpi_{n-k}v^k\bigg\|^2,
\end{equation}
where $a_j=\sum_{k=0}^{j}\varpi_k$.
\end{lemma}
\begin{proof}
Denote by $\nabla_\tau v^k:=v^k-v^{k-1}$ for $k\geq 1$.
Then, $v^k=\sum_{j=1}^{k}\nabla_\tau v^j$ and
\[
\sum_{k=1}^{n}\varpi_{n-k}v^k=
\sum_{k=1}^{n}\varpi_{n-k}\sum_{j=1}^{k}\nabla_\tau v^j
=\sum_{j=1}^{n}\nabla_\tau v^j\sum_{k=j}^{n}\varpi_{n-k}
=\sum_{j=1}^{n}a_{n-j}\nabla_\tau v^j.
\]
By Lemma \ref{lem.4} and the fact $\sum_{k=1}^{\infty}\varpi_k=0$, one can directly get
\begin{equation}\label{c3.2}
a_0\geq a_1\geq a_2\geq \cdots \geq a_{n-1}>0,\quad \text{for any } n\geq 1,
\end{equation}
followed by the inequality (\ref{c3.1}) according to Lemma A.1 of \cite{liao2019discrete}.
\end{proof}
\begin{remark}\label{rem.4}
For the sequence $\{a_j\}_{j=0}^\infty$ defined in the above lemma, there holds
\[
a(\zeta)=\sum_{j=0}^{\infty}a_j\zeta^j
=\sum_{j=0}^{\infty}\bigg(\sum_{k=0}^{j}\varpi_k\bigg)\zeta^j
=\frac{\varpi(\zeta)}{1-\zeta}
=\frac{1}{\omega(\zeta)}
=(1-\zeta)^{-\alpha}\bigg[\frac{1}{2}(1+\zeta)+\frac{\theta}{\alpha}(1-\zeta)\bigg]^\alpha.
\]
By quite similar analysis as that of Lemma \ref{lem.3}, we can get (i) $a_k=O(k^{\alpha-1})$ and (ii) $\tau^\alpha e^{-\theta\tau}a(e^{-\tau})=1+O(\tau^2)$.
If we define $\mathcal{I}^{\alpha,\theta}_\tau u^{n}:=\tau^\alpha \sum_{k=0}^{n} a_k u^{n-k}$, then (i) and (ii) indicate that
$\mathcal{I}^{\alpha,\theta}_\tau u^{n}$ approximates $\mathcal{I}^\alpha u$ at time $t=t_{n+\theta}$ with second-order accuracy (see \cite[Theorem 1]{liu2019unified}).
\end{remark}

The time semi-discrete $\theta$-scheme for (\ref{a1})-(\ref{a3}) then reads: For given $\boldsymbol E^0, H^0$ and $\boldsymbol P^0$, find $\boldsymbol E_\tau^n \in H_0({\rm curl};\Omega)\cap (L^2(\Omega))^2$, $H^n_\tau \in L^2(\Omega)$ and $\boldsymbol P^n_\tau \in (L^2(\Omega))^2$ with $n\geq 1$ such that
\begin{eqnarray}\label{c4}
\epsilon_0 \epsilon_\infty \big(\partial_\tau^{n-\theta} \boldsymbol E_\tau,\boldsymbol\chi\big)
+\big(\partial_\tau^{n-\theta} \boldsymbol P_\tau,\boldsymbol\chi\big)
-\big(\overline{H}^{n-\theta}_\tau,\nabla \times \boldsymbol\chi\big)&=&0,
\quad \forall \boldsymbol\chi \in H_0({\rm curl};\Omega),
\\
\label{c5}
\mu_0 \big(\partial_\tau^{n-\theta} H_\tau,\phi\big)
+\big(\nabla \times \overline{\boldsymbol E}_\tau^{n-\theta},\phi\big)&=&0,
\quad \forall \phi \in L^2(\Omega),\\
\label{c6}
\tau_0^\alpha \big(\partial^{\alpha}_\tau \boldsymbol P_\tau^{n-\theta},\boldsymbol\psi\big)+\big(\overline{\boldsymbol P}^{n-\theta}_\tau,\boldsymbol\psi\big)
-\epsilon_0(\epsilon_s-\epsilon_\infty) \big(\overline{\boldsymbol E}^{n-\theta}_\tau,\boldsymbol\psi\big)&=&0,
\quad \forall \boldsymbol\psi \in (L^2(\Omega))^2.
\end{eqnarray}
\begin{theorem}\label{thm1}
For arbitrary $\theta \in [\frac{\alpha}{2},\frac{1}{2}]$ with $\alpha \in (0,1)$ and any $\tau>0$, the semidiscrete scheme (\ref{c4})-(\ref{c6}) fulfills the discrete energy dissipation:
\begin{equation}\label{c7}
\widetilde{\mathcal{E}}_\alpha^n \leq
\widetilde{\mathcal{E}}_\alpha^{n-1},\quad\text{for } n\geq 1,
\end{equation}
where $\widetilde{\mathcal{E}}_\alpha^n:=\tau_0^\alpha\mathcal{I}_\tau^{\alpha,\theta} \big\|\partial_\tau^\alpha \boldsymbol P^{n-\theta}_\tau\big\|^2+\|\boldsymbol P^n_\tau\|^2+\epsilon_0(\epsilon_s-\epsilon_\infty)\big(\epsilon_0\epsilon_\infty\|\boldsymbol E^n_\tau\|^2+\mu_0 \|H^n_\tau\|^2\big)$ is the discrete counterpart of the continuous energy $\mathcal{E}_\alpha(t_n)$.
\end{theorem}
\begin{proof}
Let $\boldsymbol \chi=\overline{\boldsymbol E}^{n-\theta}_\tau$ in (\ref{c4}) and $\phi=\overline{H}^{n-\theta}_\tau$ in (\ref{c5}), respectively, and add these two equation together to obtain
\begin{equation}\label{c7}
\frac{1}{2}\epsilon_0\epsilon_\infty\partial_\tau^{n-\theta}\|\boldsymbol E_\tau\|^2
+\frac{\mu_0}{2}\partial_\tau^{n-\theta}\|H_\tau\|^2
+\big(\partial_\tau^{n-\theta}\boldsymbol P_\tau,\overline{\boldsymbol E}^{n-\theta}_\tau\big)\leq 0,
\end{equation}
where we have used the fact that $(a-b)((1-\theta)a+\theta b)\geq \frac{1}{2}(a^2-b^2)$ provided $\theta \in (0,\frac{1}{2}]$.
Replace $\boldsymbol\psi$ in (\ref{c6}) by $\partial_\tau^{n-\theta}\boldsymbol P_\tau$ and add the result to (\ref{c7}) after multiplying (\ref{c7}) by $\epsilon_0(\epsilon_s-\epsilon_\infty)$, to get
\begin{equation}\label{c8}
\tau_0^\alpha \big(\partial_\tau^\alpha \boldsymbol P^{n-\theta}_\tau,\partial_\tau^{n-\theta}\boldsymbol P_\tau\big)
+\frac{1}{2}\partial_\tau^{n-\theta}\|\boldsymbol P_\tau\|^2
+\frac{1}{2}\epsilon_0(\epsilon_s-\epsilon_\infty)\big(\epsilon_0\epsilon_\infty\partial_\tau^{n-\theta}\|\boldsymbol E_\tau\|^2+\mu_0\partial_\tau^{n-\theta}\|H_\tau\|^2\big)\leq 0.
\end{equation}
then, using the weights $\varpi_k$ one can immediately obtain
\begin{equation}\label{c9}
  \partial_\tau^{n-\theta}\boldsymbol P_\tau=\tau^{\alpha-1}\sum_{k=1}^{n}\varpi_{n-k}\big(\partial_\tau^\alpha \boldsymbol P^{k-\theta}_\tau\big).
\end{equation}
By Lemma \ref{lem.5}, the first term of (\ref{c8}) reads
\begin{equation}\label{c10}\begin{split}
&\big(\partial_\tau^\alpha \boldsymbol P^{n-\theta}_\tau,\partial_\tau^{n-\theta}\boldsymbol P_\tau\big)
=\tau^{\alpha-1}\bigg(\partial_\tau^\alpha \boldsymbol P^{n-\theta}_\tau,\sum_{k=1}^{n}\varpi_{n-k}\big(\partial_\tau^\alpha \boldsymbol P^{k-\theta}_\tau\big)\bigg)
\\&\geq
\frac{\tau^{\alpha-1}}{2}\sum_{k=1}^{n}a_{n-k}\big[\big\|\partial_\tau^\alpha \boldsymbol P^{k-\theta}_\tau\big\|^2-\big\|\partial_\tau^\alpha \boldsymbol P^{k-1-\theta}_\tau\big\|^2\big]
+\frac{\tau^{\alpha-1}}{2\varpi_0}\bigg\|\sum_{k=1}^{n}\varpi_{n-k}\big(\partial_\tau^\alpha \boldsymbol P^{k-\theta}_\tau\big)\bigg\|^2
\\&=
\frac{\tau^{\alpha-1}}{2}\sum_{k=1}^{n}a_{n-k}\big\|\partial_\tau^\alpha \boldsymbol P^{k-\theta}_\tau\big\|^2
-\frac{\tau^{\alpha-1}}{2}\sum_{k=1}^{n-1}a_{n-1-k}\big\|\partial_\tau^\alpha \boldsymbol P^{k-\theta}_\tau\big\|^2
+\frac{\tau^{\alpha-1}}{2\varpi_0}\bigg\|\sum_{k=1}^{n}\varpi_{n-k}\big(\partial_\tau^\alpha \boldsymbol P^{k-\theta}_\tau\big)\bigg\|^2
\\&=
\frac{1}{2\tau}\big(\mathcal{I}_\tau^{\alpha,\theta} \big\|\partial_\tau^\alpha \boldsymbol P^{n-\theta}_\tau\big\|^2
-\mathcal{I}_\tau^{\alpha,\theta} \big\|\partial_\tau^\alpha \boldsymbol P^{n-1-\theta}_\tau\big\|^2\big)
+\frac{\tau^{\alpha-1}}{2\varpi_0}\bigg\|\sum_{k=1}^{n}\varpi_{n-k}\big(\partial_\tau^\alpha \boldsymbol P^{k-\theta}_\tau\big)\bigg\|^2.
\end{split}\end{equation}
Finally, combining (\ref{c8}), (\ref{c9}) with (\ref{c10}), one can get
\begin{equation}\label{c11}
 \partial_\tau^{n-\theta}\widetilde{\mathcal{E}}_\alpha
 +\tau_0^\alpha \frac{\tau^{1-\alpha}}{\varpi_0}
 \big\|
 \partial_\tau^{n-\theta}\boldsymbol P_\tau
 \big\|^2 \leq 0,
\end{equation}
meaning that $\partial_\tau^{n-\theta}\widetilde{\mathcal{E}}_\alpha \leq 0$, which completes the proof of the theorem.
\end{proof}

\section{Error Analysis}
In this section, we consider the convergence rate of the semi-discrete scheme by assuming the solution is smooth enough.
Source terms are added to the Cole-Cole model, reading that
\begin{eqnarray}\label{aa1}
\epsilon_0 \epsilon_\infty \frac{\partial \boldsymbol E}{\partial t}(\boldsymbol x,t)&=&\nabla \times H(\boldsymbol x,t)-\frac{\partial \boldsymbol P}{\partial t}(\boldsymbol x,t)+\boldsymbol f_1(\boldsymbol x,t),\\
\label{aa2}
\mu_0 \frac{\partial H}{\partial t}(\boldsymbol x,t)&=&-\nabla \times \boldsymbol E(\boldsymbol x,t)+ f_2(\boldsymbol x,t),\\
\label{aa3}
\tau_0^\alpha \partial^\alpha_t \boldsymbol P(\boldsymbol x,t)+\boldsymbol P(\boldsymbol x,t)&=&\epsilon_0(\epsilon_s-\epsilon_\infty) \boldsymbol E(\boldsymbol x,t)+\boldsymbol f_3(\boldsymbol x,t).
\end{eqnarray}
The time semi-discrete scheme is formulated as
\begin{eqnarray}\label{cc4}
\epsilon_0 \epsilon_\infty \partial_\tau^{n-\theta} \boldsymbol E_\tau
+\partial_\tau^{n-\theta} \boldsymbol P_\tau
-\nabla \times \overline{H}^{n-\theta}_\tau&=&\boldsymbol f_1(\boldsymbol x,t_{n-\theta}),
\\
\label{cc5}
\mu_0 \partial_\tau^{n-\theta} H_\tau
+\nabla \times \overline{\boldsymbol E}_\tau^{n-\theta}&=&f_2(\boldsymbol x,t_{n-\theta}),\\
\label{cc6}
\tau_0^\alpha \partial^{\alpha}_\tau \boldsymbol P_\tau^{n-\theta}
+\overline{\boldsymbol P}^{n-\theta}_\tau
-\epsilon_0(\epsilon_s-\epsilon_\infty) \overline{\boldsymbol E}^{n-\theta}_\tau&=&\boldsymbol f_3(\boldsymbol x,t_{n-\theta}).
\end{eqnarray}
The following theorem demonstrates that the temporal accuracy of the energy-dissipation-preserving scheme is of order $O(\tau)$ when $\theta \ne \frac{1}{2}$, and improves to $O(\tau^2)$ when $\theta = \frac{1}{2}$.
\begin{theorem}
Let $(\boldsymbol E(t),\boldsymbol P(t), H(t))$ be the solution of (\ref{a1})-(\ref{a3}) where $\boldsymbol E \in C(0,T; H_0({\rm curl};\Omega))\cap C^3(0,T;(L^2(\Omega))^2)$, $\boldsymbol P \in C^3(0,T;(L^2(\Omega))^2)$ and $H\in C^3(0,T;H_0({\rm curl};\Omega))$.
Let $(\boldsymbol E^n_\tau,\boldsymbol P^n_\tau, H^n_\tau)$ be the solution of the time semi-discrete scheme (\ref{c4})-(\ref{c6}).
For sufficiently small $\tau$ there holds that
\begin{equation}\label{F.9}\begin{split}
\max_{n\geq 1}\big(\|\boldsymbol E^n-\boldsymbol E_\tau^n\|
+\|\boldsymbol P^n-\boldsymbol P_\tau^n\|
+\|H^n-H_\tau^n\|\big)
\leq C\tau^s,\quad
s=
\begin{cases}
1, & \text{if ~} \theta \in [\frac{\alpha}{2},\frac{1}{2}),\\
2, & \text{if ~} \theta=\frac{1}{2},
\end{cases}
\end{split}
\end{equation}
 where $C$ is a constant independent of $\tau$.
 \begin{proof}
 Rewrite the Cole-Cole model with source terms (\ref{aa1})-(\ref{aa3}) into the following form
 \begin{eqnarray}\label{cdc4}
\epsilon_0 \epsilon_\infty \partial_\tau^{n-\theta} \boldsymbol E
+\partial_\tau^{n-\theta} \boldsymbol P
-\nabla \times \overline{H}^{n-\theta}&=&\boldsymbol f_1(\boldsymbol x,t_{n-\theta})+\boldsymbol R_1^{n-\theta},
\\
\label{cdc5}
\mu_0 \partial_\tau^{n-\theta} H
+\nabla \times \overline{\boldsymbol E}^{n-\theta}&=&f_2(\boldsymbol x,t_{n-\theta})+R_2^{n-\theta},\\
\label{cdc6}
\tau_0^\alpha \partial^{\alpha}_\tau \boldsymbol P^{n-\theta}
+\overline{\boldsymbol P}^{n-\theta}
-\epsilon_0(\epsilon_s-\epsilon_\infty) \overline{\boldsymbol E}^{n-\theta}&=&\boldsymbol f_3(\boldsymbol x,t_{n-\theta})+\boldsymbol R_3^{n-\theta},
\end{eqnarray}
where the error terms are defined by
\begin{equation}\begin{split}
\boldsymbol R_1^{n-\theta}
&=\epsilon_0 \epsilon_\infty \bigg(\partial_\tau^{n-\theta} \boldsymbol E-\frac{\partial \boldsymbol E}{\partial t}(\boldsymbol x,t_{n-\theta})\bigg)
+\bigg(\partial_\tau^{n-\theta} \boldsymbol P-\frac{\partial \boldsymbol P}{\partial t}(\boldsymbol x,t_{n-\theta})\bigg)
\\&
\quad -\nabla \times\big( \overline{H}^{n-\theta}- H(\boldsymbol x,t_{n-\theta})\big),
\\
R_2^{n-\theta}&=\mu_0 \bigg(\partial_\tau^{n-\theta} H-\frac{\partial H}{\partial t}(\boldsymbol x,t_{n-\theta})\bigg)
+\nabla \times \big(\overline{\boldsymbol E}^{n-\theta}- \boldsymbol E(\boldsymbol x,t_{n-\theta})\big),
\\
\boldsymbol R_3^{n-\theta}&=
\tau_0^\alpha \big(\partial^{\alpha}_\tau \boldsymbol P^{n-\theta}-\partial^\alpha_t \boldsymbol P(\boldsymbol x,t_{n-\theta})\big)
-\epsilon_0(\epsilon_s-\epsilon_\infty) \big(\overline{\boldsymbol E}^{n-\theta}-\boldsymbol E(\boldsymbol x,t_{n-\theta})\big).
\end{split}\end{equation}
It is clear that $\|\boldsymbol R_1^{n-\theta}\|+\|R_2^{n-\theta}\|+\|\boldsymbol R_3^{n-\theta}\|=O(\tau^s)$ where $s=1$ if $\theta \neq \frac{1}{2}$ and $s=2$ if $\theta=\frac{1}{2}$.
Let ${\boldsymbol\xi}_E^n:=\boldsymbol E^n-\boldsymbol E^n_\tau$, $\xi_H^n:=H^n-H_\tau^n$ and ${\boldsymbol\xi}_P^n:=\boldsymbol P^n-\boldsymbol P_\tau^n$.
Subtracting (\ref{cc4})-(\ref{cc6}) from (\ref{cdc4})-(\ref{cdc6}), one gets
\begin{eqnarray}\label{dd4}
\epsilon_0 \epsilon_\infty \partial_\tau^{n-\theta} {\boldsymbol\xi}_E
+\partial_\tau^{n-\theta} {\boldsymbol\xi}_P
-\nabla \times \overline{\xi}_H^{n-\theta}&=&\boldsymbol R_1^{n-\theta},
\\
\label{dd5}
\mu_0 \partial_\tau^{n-\theta} \xi_H
+\nabla \times \overline{\boldsymbol \xi}_E^{n-\theta}&=&R_2^{n-\theta},\\
\label{dd6}
\tau_0^\alpha \partial^{\alpha}_\tau \boldsymbol \xi_P^{n-\theta}
+\overline{\boldsymbol \xi}_P^{n-\theta}
-\epsilon_0(\epsilon_s-\epsilon_\infty) \overline{\boldsymbol \xi}_E^{n-\theta}&=&\boldsymbol R_3^{n-\theta}.
\end{eqnarray}
By similar arguments as the stability analysis in Theorem \ref{thm1}, we have
\begin{equation}\begin{split}\label{b7}
&\quad\tau^{1-\alpha}
 \big\|
 \partial_\tau^{n-\theta}\boldsymbol \xi_P
 \big\|^2
+\partial_\tau^{n-\theta}\|\boldsymbol \xi_P\|^2
+\partial_\tau^{n-\theta}\|\boldsymbol \xi_E\|^2
+\partial_\tau^{n-\theta}\|\xi_H\|^2
\\ &\leq 
C (\boldsymbol R_1^{n-\theta},\overline{\boldsymbol \xi}_E^{n-\theta})
+C(R_2^{n-\theta},\overline{\xi}_H^{n-\theta})
+C(\boldsymbol R_3^{n-\theta},\partial_\tau^{n-\theta} {\boldsymbol\xi}_P).
\end{split}\end{equation}
Multiply both hand-side of (\ref{b7}) by $\tau$, replace $n$ with $j$ and sum the index $j$ from $1$ to $n$ to obtain
\begin{equation}\begin{split}\label{b8}
&\quad
\|\boldsymbol \xi_P^n\|^2
+\|\boldsymbol \xi_E^n\|^2
+\|\xi_H^n\|^2
\\&\leq 
C\tau\sum_{j=1}^n
\bigg(
 (\boldsymbol R_1^{j-\theta},\overline{\boldsymbol \xi}_E^{j-\theta})
+(R_2^{j-\theta},\overline{\xi}_H^{j-\theta})
+(\boldsymbol R_3^{j-\theta},\partial_\tau^{j-\theta} {\boldsymbol\xi}_P)
\bigg).
\end{split}\end{equation}
Based on the identity
\begin{equation}\begin{split}\label{b9}
(\boldsymbol R_3^{j-\theta},\partial_\tau^{j-\theta} {\boldsymbol\xi}_P)
=\frac{1}{\tau}\big[
(\boldsymbol R_3^{j-\theta},{\boldsymbol\xi}_P^j)
-(\boldsymbol R_3^{j-1-\theta},{\boldsymbol\xi}_P^{j-1})
\big]
-(\partial_{\tau,\theta}^j \boldsymbol R_3,{\boldsymbol\xi}_P^{j-1})
\end{split}\end{equation}
where $\partial_{\tau,\theta}^j \boldsymbol R_3:=\frac{1}{\tau}( \boldsymbol R_3^{j-\theta}-\boldsymbol R_3^{j-\theta-1})$.
For sufficiently smooth $\boldsymbol P$ and $\boldsymbol E$, it holds that $\|\partial_{\tau,\theta}^j \boldsymbol R_3\|=O(\tau^s)$.
Applying the Cauchy-Schwarz and Young inequalities, (\ref{b8}) and (\ref{b9}) yield
\begin{equation}\begin{split}
\|\boldsymbol \xi_P^n\|^2
+\|\boldsymbol \xi_E^n\|^2
+\|\xi_H^n\|^2
\leq
C\tau \sum_{j=1}^n
\big(
\|\boldsymbol \xi_P^j\|^2
+\|\boldsymbol \xi_E^j\|^2
+\|\xi_H^j\|^2
\big)
+O(\tau^{2s}).
\end{split}\end{equation}
Finally, using the Gronwall inequality, we complete the proof of the theorem.
 \end{proof}
\end{theorem}
\begin{remark}
The fully discrete scheme can be directly formulated using the \textit{N\'{e}d\'{e}lec} element, and its error analysis follows straightforwardly. For completeness, we refer the reader to \cite{baoli2023discrete} for details, as these aspects fall outside the scope of this paper.
\end{remark}

\section{Numerical tests}
In this section, we perform numerical tests to validate the theoretical findings regarding the energy-decay property and convergence rates.  
\subsection{C-C model with source terms}
We consider the domain $\Omega=(0,1)^2$ and final time $T=1$, with the parameter values specified by $\epsilon_0\epsilon_\infty=\mu_0=\tau_0=\epsilon_0(\epsilon_s-\epsilon_\infty)=1$. For the purpose of convergence verification, the following smooth functions are prescribed as the exact solutions:
\begin{equation}\label{N.1}\begin{split}
\boldsymbol E(x,y,t)&=e^{-t}
\begin{pmatrix}
  (x^2+1)\sin(\pi y) \\
  \sin(\pi x) (y-\frac{1}{2})
\end{pmatrix},\quad
\boldsymbol P(x,y,t)=t^3
\begin{pmatrix}
  (x^2+1)y(y-1) \\
   x(x-1)(y-\frac{1}{2})
\end{pmatrix},
\\
H(x,y,t)&=e^{-t}(x^3+1)(y^3+1).
\end{split}
\end{equation}
The corresponding source terms are subsequently derived from these defined solutions.
\par
As a comparative benchmark, we approximate the term \(\partial^\alpha_t \boldsymbol{P}(\boldsymbol{x}, t_{n-\theta})\) using alternative approaches as follows:
\begin{equation}\begin{split}\label{t1}
\partial^\alpha_t \boldsymbol P(\boldsymbol x,t_{n-\theta})&=
(1-\theta)\partial^\alpha_t \boldsymbol P(\boldsymbol x,t_{n})
+\theta\partial^\alpha_t \boldsymbol P(\boldsymbol x,t_{n-1})
+O(\tau^2)
\\&=
\tau^{-\alpha}\sum_{k=0}^n
[(1-\theta)\tilde{\omega}_{n-k}+\theta\tilde{\omega}_{n-k-1}]
 \boldsymbol P(\boldsymbol x,t_{k})+O(\tau^2),
\end{split}\end{equation}
where the coefficients \(\tilde{\omega}_k\) are generated from the fractional BDF-2 scheme \cite{lubich1986discretized} via the generating function \(\tilde{\omega}(\zeta) = \left( \frac{3}{2} - 2\zeta + \frac{1}{2} \zeta^2 \right)^\alpha\).
\par
For notational clarity, we refer to the numerical scheme proposed in this paper as \textbf{SFTR-\(\theta\)}, and denote the alternative scheme based on approximation (\ref{t1}) as \textbf{F-BDF-2}.
Define the error at time step \(n\) for \( \boldsymbol{E} \) as \(\text{Error}_{\boldsymbol{E}^n} := \|\boldsymbol{E}^n - \boldsymbol{E}^n_\tau\|\). The corresponding global error in time is then given by \(\text{Error}_{\boldsymbol{E}} := \max_{n} \text{Error}_{\boldsymbol{E}^n}\). The quantities \( \boldsymbol{P} \) and \( H \) are treated analogously.
\par
The numerical results presented in Tables \ref{tab1} and \ref{tab2} validate the temporal convergence properties of two numerical schemes for solving the C-C model: the SFTR-$\theta$ method (Table \ref{tab1}) and the F-BDF-2 method (Table \ref{tab2}). Both tables display the global errors and corresponding convergence rates for the electric field ($\boldsymbol E$), magnetic field (H), and polarization ($\boldsymbol P$) with a fixed spatial discretization \(h = \sqrt{2}/100\) and varying time steps $\tau$.
The data confirm the theoretical convergence orders for both schemes. When $\theta \neq 0.5$, first-order convergence is generally observed, particularly for smaller values of $\alpha$. When $\theta=0.5$, both methods achieve second-order convergence as expected. For $\theta \neq 0.5$, both schemes occasionally exhibit convergence rates exceeding the theoretical first-order prediction, particularly at larger $\alpha$ values.
\begin{table}[]
\centering
\caption{Error and convergence rates for the SFTR-$\theta$ scheme with $h=\frac{\sqrt{2}}{100}$.}\label{tab1}
{\renewcommand{\arraystretch}{1.2}
\begin{tabular}{cclcccccc}
\hline
$\alpha$             & $\theta$              & $\tau$ & $\text{Error}_{\boldsymbol{E}}$ & Rates & $\text{Error}_{H}$ & Rates & \begin{tabular}[c]{@{}c@{}}$\text{Error}_{\boldsymbol{P}}$\end{tabular} & Rates \\ \hline
\multirow{8}{*}{0.1} & \multirow{4}{*}{0.05} & 1/5    & 1.3051E-02                                       &       & 8.6135E-02                &       & 6.4191E-03                                                                                     &       \\
                     &                       & 1/10   & 7.5273E-03                                       & 0.79  & 4.3739E-02                & 0.98  & 3.5507E-03                                                                                     & 0.85  \\
                     &                       & 1/20   & 4.0212E-03                                       & 0.90  & 2.2044E-02                & 0.99  & 1.8611E-03                                                                                     & 0.93  \\
                     &                       & 1/40   & 2.0713E-03                                       & 0.96  & 1.1069E-02                & 0.99  & 9.5010E-04                                                                                     & 0.97  \\ \cline{3-9} 
                     & \multirow{4}{*}{0.5}  & 1/5    & 7.7334E-03                                       &       & 1.2382E-02                &       & 1.2371E-02                                                                                     &       \\
                     &                       & 1/10   & 2.4230E-03                                       & 1.67  & 4.1740E-03                & 1.57  & 4.7133E-03                                                                                     & 1.39  \\
                     &                       & 1/20   & 6.9085E-04                                       & 1.81  & 1.2611E-03                & 1.73  & 1.5101E-03                                                                                     & 1.64  \\
                     &                       & 1/40   & 1.8269E-04                                       & 1.92  & 3.4631E-04                & 1.86  & 4.2948E-04                                                                                     & 1.81  \\ \hline
\multirow{8}{*}{0.5} & \multirow{4}{*}{0.25} & 1/5    & 9.0433E-03                                       &       & 4.6480E-02                &       & 4.6527E-03                                                                                     &       \\
                     &                       & 1/10   & 4.9175E-03                                       & 0.88  & 2.3967E-02                & 0.96  & 1.9206E-03                                                                                     & 1.28  \\
                     &                       & 1/20   & 2.5560E-03                                       & 0.94  & 1.2169E-02                & 0.98  & 8.7310E-04                                                                                     & 1.14  \\
                     &                       & 1/40   & 1.3006E-03                                       & 0.97  & 6.1315E-03                & 0.99  & 4.1816E-04                                                                                     & 1.06  \\ \cline{3-9} 
                     & \multirow{4}{*}{0.5}  & 1/5    & 2.9735E-03                                       &       & 6.0447E-03                &       & 1.0138E-03                                                                                     &       \\
                     &                       & 1/10   & 7.6690E-04                                       & 1.96  & 1.5668E-03                & 1.95  & 3.5522E-04                                                                                     & 1.51  \\
                     &                       & 1/20   & 1.9182E-04                                       & 2.00  & 3.9566E-04                & 1.99  & 1.0197E-04                                                                                     & 1.80  \\
                     &                       & 1/40   & 4.8005E-05                                       & 2.00  & 9.9517E-05                & 1.99  & 2.7168E-05                                                                                     & 1.91  \\ \hline
\multirow{8}{*}{0.9} & \multirow{4}{*}{0.45} & 1/5    & 2.9562E-03                                       &       & 7.1173E-03                &       & 3.8101E-03                                                                                     &       \\
                     &                       & 1/10   & 1.3316E-03                                       & 1.15  & 4.2064E-03                & 0.76  & 9.7433E-04                                                                                     & 1.97  \\
                     &                       & 1/20   & 6.3306E-04                                       & 1.07  & 2.2842E-03                & 0.88  & 2.6404E-04                                                                                     & 1.88  \\
                     &                       & 1/40   & 3.0873E-04                                       & 1.04  & 1.1893E-03                & 0.94  & 8.5076E-05                                                                                     & 1.63  \\ \cline{3-9} 
                     & \multirow{4}{*}{0.5}  & 1/5    & 1.6617E-03                                       &       & 4.1652E-03                &       & 3.1349E-03                                                                                     &       \\
                     &                       & 1/10   & 4.1228E-04                                       & 2.01  & 1.0460E-03                & 1.99  & 7.7408E-04                                                                                     & 2.02  \\
                     &                       & 1/20   & 1.0253E-04                                       & 2.01  & 2.6175E-04                & 2.00  & 1.9240E-04                                                                                     & 2.01  \\
                     &                       & 1/40   & 2.5715E-05                                       & 2.00  & 6.5960E-05                & 1.99  & 4.7989E-05                                                                                     & 2.00  \\ \hline
\end{tabular}
}
\end{table}

\begin{table}[]
\centering
\caption{Error and convergence rates for the F-BDF-2 method with $h=\frac{\sqrt{2}}{100}$.}\label{tab2}
{\renewcommand{\arraystretch}{1.2}
\begin{tabular}{cclcccccc}
\hline
$\alpha$             & $\theta$              & $\tau$ & $\text{Error}_{\boldsymbol{E}}$ & Rates & $\text{Error}_{H}$ & Rates & $\text{Error}_{\boldsymbol{P}}$ & Rates \\ \hline
\multirow{8}{*}{0.1} & \multirow{4}{*}{0.05} & 1/5    & 1.2990E-02                      &       & 8.6118E-02         &       & 6.6458E-03                      &       \\
                     &                       & 1/10   & 7.5131E-03                      & 0.79  & 4.3736E-02         & 0.98  & 3.6152E-03                      & 0.88  \\
                     &                       & 1/20   & 4.0179E-03                      & 0.90  & 2.2044E-02         & 0.99  & 1.8779E-03                      & 0.94  \\
                     &                       & 1/40   & 2.0705E-03                      & 0.96  & 1.1069E-02         & 0.99  & 9.5437E-04                      & 0.98  \\ \cline{3-9} 
                     & \multirow{4}{*}{0.5}  & 1/5    & 2.0542E-03                      &       & 4.7827E-03         &       & 1.6645E-03                      &       \\
                     &                       & 1/10   & 5.0657E-04                      & 2.02  & 1.1957E-03         & 2.00  & 4.3319E-04                      & 1.94  \\
                     &                       & 1/20   & 1.2693E-04                      & 2.00  & 2.9866E-04         & 2.00  & 1.1051E-04                      & 1.97  \\
                     &                       & 1/40   & 3.1827E-05                      & 2.00  & 7.5069E-05         & 1.99  & 2.7951E-05                      & 1.98  \\ \hline
\multirow{8}{*}{0.5} & \multirow{4}{*}{0.25} & 1/5    & 8.5979E-03                      &       & 4.6331E-02         &       & 6.7442E-03                      &       \\
                     &                       & 1/10   & 4.8414E-03                      & 0.83  & 2.3936E-02         & 0.95  & 2.4456E-03                      & 1.46  \\
                     &                       & 1/20   & 2.5424E-03                      & 0.93  & 1.2162E-02         & 0.98  & 9.9213E-04                      & 1.30  \\
                     &                       & 1/40   & 1.2980E-03                      & 0.97  & 6.1302E-03         & 0.99  & 4.4466E-04                      & 1.16  \\ \cline{3-9} 
                     & \multirow{4}{*}{0.5}  & 1/5    & 8.7112E-04                      &       & 3.4754E-03         &       & 5.2863E-03                      &       \\
                     &                       & 1/10   & 1.8965E-04                      & 2.20  & 8.5412E-04         & 2.02  & 1.4162E-03                      & 1.90  \\
                     &                       & 1/20   & 5.1086E-05                      & 1.89  & 2.1277E-04         & 2.01  & 3.6668E-04                      & 1.95  \\
                     &                       & 1/40   & 1.3673E-05                      & 1.90  & 5.3786E-05         & 1.98  & 9.3296E-05                      & 1.97  \\ \hline
\multirow{8}{*}{0.9} & \multirow{4}{*}{0.45} & 1/5    & 2.1597E-03                      &       & 6.6744E-03         &       & 9.7369E-03                      &       \\
                     &                       & 1/10   & 1.2583E-03                      & 0.78  & 4.1196E-03         & 0.70  & 2.5996E-03                      & 1.91  \\
                     &                       & 1/20   & 6.3639E-04                      & 0.98  & 2.2642E-03         & 0.86  & 6.7628E-04                      & 1.94  \\
                     &                       & 1/40   & 3.1195E-04                      & 1.03  & 1.1846E-03         & 0.93  & 1.7900E-04                      & 1.92  \\ \cline{3-9} 
                     & \multirow{4}{*}{0.5}  & 1/5    & 9.4455E-04                      &       & 3.7415E-03         &       & 9.7657E-03                      &       \\
                     &                       & 1/10   & 3.1100E-04                      & 1.60  & 9.7887E-04         & 1.93  & 2.6054E-03                      & 1.91  \\
                     &                       & 1/20   & 8.7175E-05                      & 1.83  & 2.4558E-04         & 1.99  & 6.7213E-04                      & 1.95  \\
                     &                       & 1/40   & 2.3919E-05                      & 1.87  & 6.1911E-05         & 1.99  & 1.7068E-04                      & 1.98  \\ \hline
\end{tabular}
}
\end{table}

\subsection{C-C model without source terms}
In this example, let $\Omega=(0,1)^2$, $T=1$, and set parameter values $\epsilon_0\epsilon_\infty=\mu_0=\tau_0=\epsilon_0(\epsilon_s-\epsilon_\infty)=1$.
Choose the following initial conditions for the initial C-C model (\ref{a1})-(\ref{a3}):
\begin{equation}\label{N.2}\begin{split}
\boldsymbol E_0(x,y)&=
\begin{pmatrix}
  (x^2+1)\sin(\pi y) \\
  \sin(\pi x) (y-\frac{1}{2})
\end{pmatrix},\quad
\boldsymbol P_0(x,y)=\boldsymbol 0,\quad
H_0(x,y)=(x^3+1)(y^3+1).
\end{split}
\end{equation}

Figs. \ref{C2} and \ref{C3} present numerical experiments investigating the energy-decay properties of the SFTR-$\theta$ scheme and F-BDF-2 method.
Fig. \ref{C2} demonstrates that the SFTR-$\theta$ scheme maintains strict monotonic energy-decay across various parameter values. Subfigure (a) shows consistent decay for $\alpha = 0.5$ with $\theta = 0.3, 0.4, 0.5$, while subfigure (b) confirms this property for $\theta = 0.5$ with $\alpha$ ranging from 0.1 to 0.9.
A smaller value of $\alpha$ leads to faster energy-decay in the vicinity of the initial time.
Fig.  \ref{C3}  provides comparative analysis for $\theta = 0.5$ at different $\alpha$ values $(0.2, 0.5, 0.8, 0.99)$. 
The SFTR-$\theta$ scheme preserves the discrete energy dissipation property $\widetilde{\mathcal{E}}_\alpha^n \leq
\widetilde{\mathcal{E}}_\alpha^{n-1}$ in all cases. In contrast, the F-BDF-2 method exhibits non-physical energy oscillations and fails to maintain monotonic decay.
\par
This distinction highlights the SFTR-$\theta$ scheme's superior capability in preserving the dissipative structure of the continuous model, ensuring numerical stability and physical fidelity in long-term simulations.

\begin{figure}[htbp]
\centering
\subfigure[]{
\begin{minipage}[t]{0.5\linewidth}
\centering
\includegraphics[width=1\textwidth]{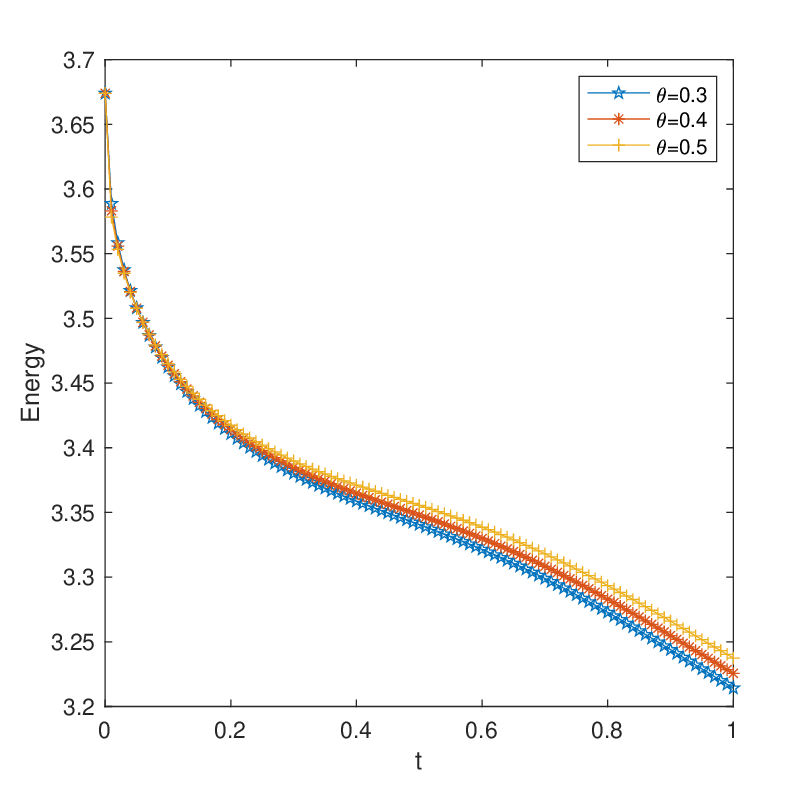}
\end{minipage}%
}%
\subfigure[]{
\begin{minipage}[t]{0.5\linewidth}
\centering
\includegraphics[width=1\textwidth]{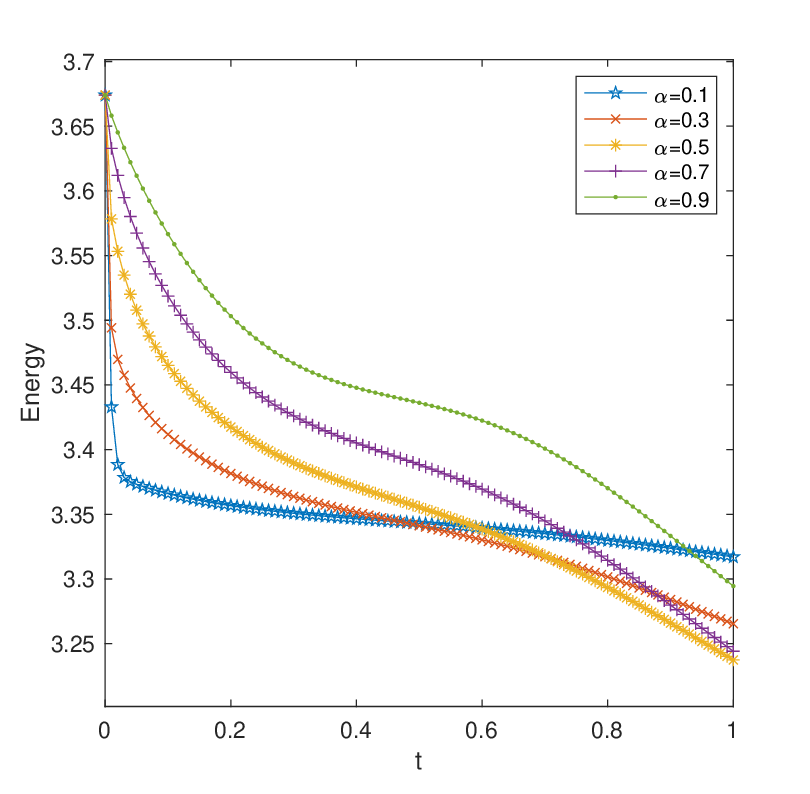}
\end{minipage}%
}%
\centering
\caption{Illustration of the discrete energy-decay property of the SFTR-\(\theta\) scheme with \(h = \sqrt{2}/60\) and \(\tau = 0.01\). (a) Energy evolution for \(\alpha = 0.5\) with varying \(\theta = 0.3, 0.4, 0.5\); (b) Energy evolution for \(\theta = 0.5\) with varying \(\alpha = 0.1, 0.3, 0.5, 0.7, 0.9\).}\label{C2}
\end{figure}

\begin{figure}[htbp]
\centering
\subfigure[]{
\begin{minipage}[t]{0.5\linewidth}
\centering
\includegraphics[width=1\textwidth]{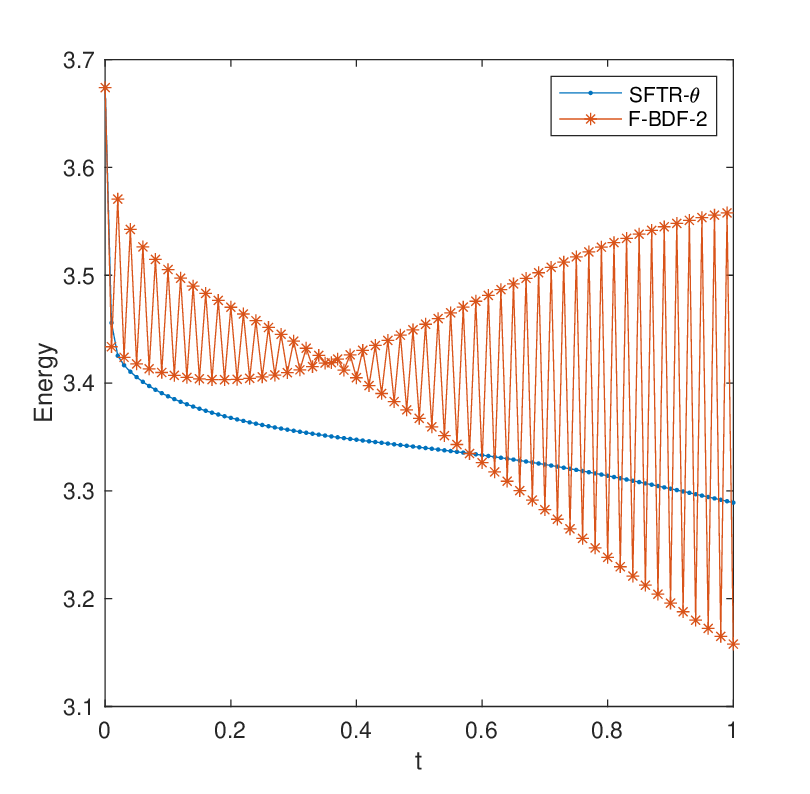}
\end{minipage}%
}%
\subfigure[]{
\begin{minipage}[t]{0.5\linewidth}
\centering
\includegraphics[width=1\textwidth]{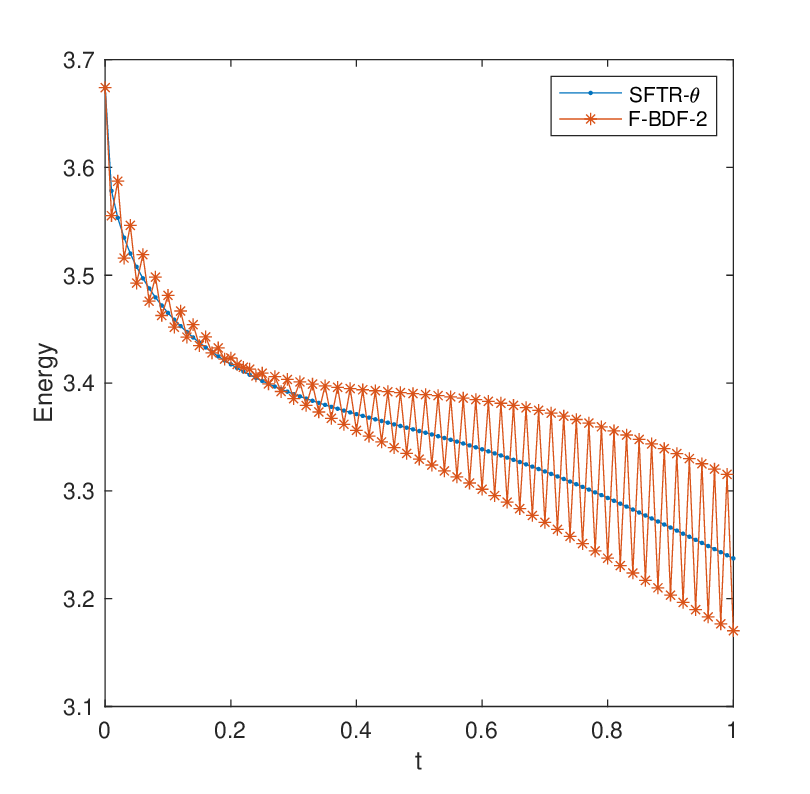}
\end{minipage}%
}%
\\
\subfigure[]{
\begin{minipage}[t]{0.5\linewidth}
\centering
\includegraphics[width=1\textwidth]{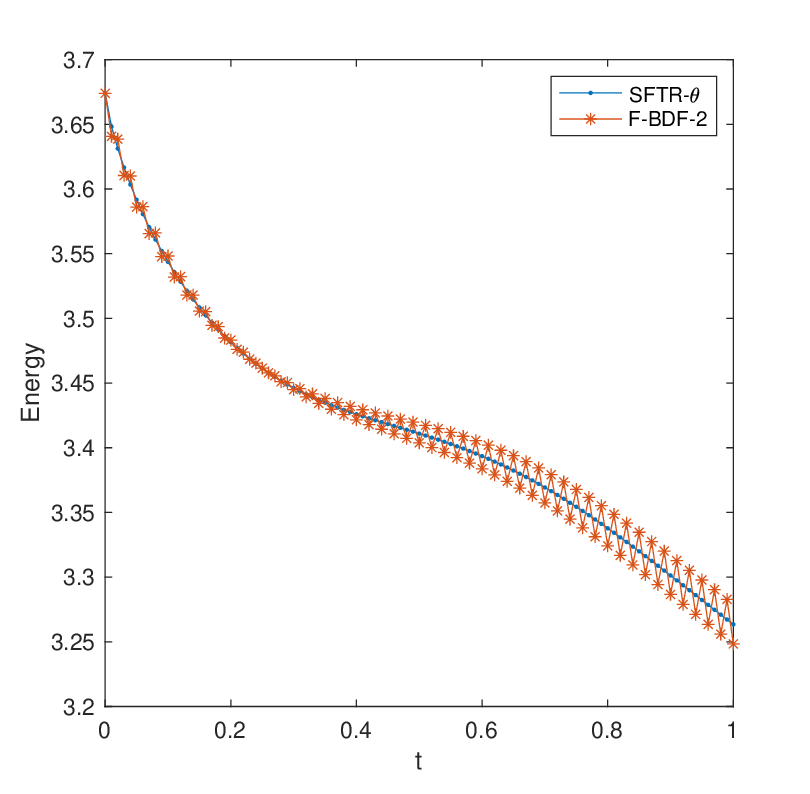}
\end{minipage}%
}%
\subfigure[]{
\begin{minipage}[t]{0.5\linewidth}
\centering
\includegraphics[width=1\textwidth]{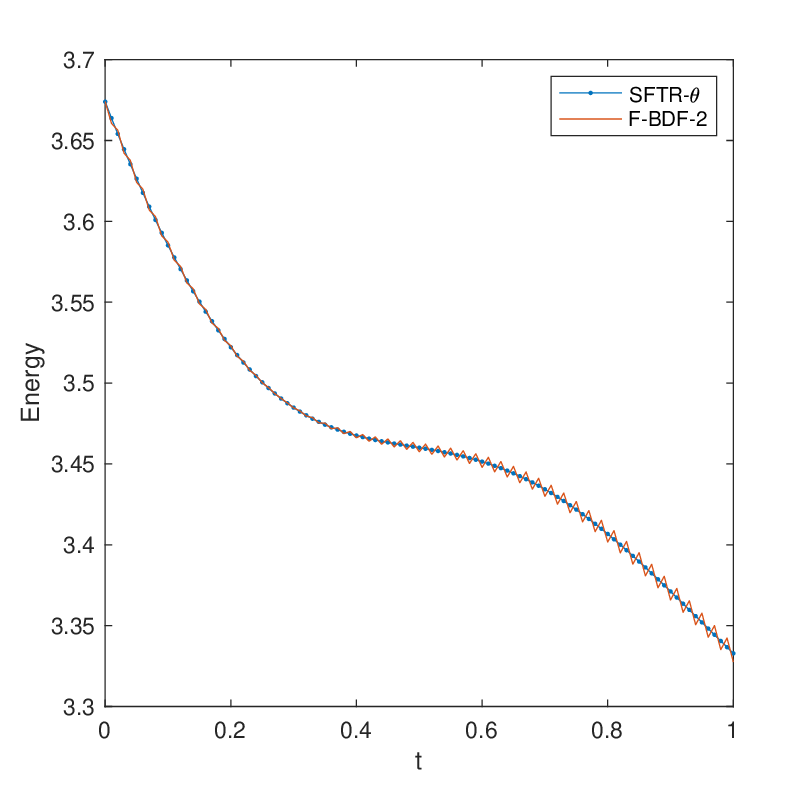}
\end{minipage}%
}%
\centering
\caption{Comparison of the discrete energy-decay between the SFTR-\(\theta\) scheme and the F-BDF-2 method, with parameters \(h = \sqrt{2}/60\), \(\tau = 0.01\), and \(\theta = 0.5\): (a) \(\alpha = 0.2\); (b) \(\alpha = 0.5\); (c) \(\alpha = 0.8\); (d) \(\alpha = 0.99\).}\label{C3}
\end{figure}
\section{Conclusion}
In this work, we have systematically studied the energy dissipation properties of the Cole–Cole model and developed a class of temporal discretization schemes that preserve discrete energy-decay. 
The main contributions are summarized as follows:
\begin{itemize}
\item[-] A modified energy functional was introduced, and a rigorous proof of the energy-decay law for the continuous C-C model was established.
\item[-] A $\theta$-scheme (SFTR-$\theta$) was designed and analyzed, with a theoretical proof that it preserves the discrete energy-decay property under appropriate parameter choices.
\item[-] A detailed error analysis confirmed that the scheme achieves first-order convergence for \(\theta \neq 0.5\) and second-order convergence for \(\theta = 0.5\).
\item[-] Numerical tests verified the theoretical convergence rates and demonstrated the unconditional energy-decay behavior of the SFTR-$\theta$ scheme. In contrast, the F-BDF-2 method failed to preserve monotonic energy-decay in various settings.
\end{itemize}
\par
The proposed scheme not only aligns with the physical dissipation structure of the C-C model but also provides a reliable foundation for adaptive time-stepping strategies in long-time simulations. Future work may extend this approach to more complex dispersive models such as the Havriliak-Negami model.
\section*{Declaration of competing interest}
The authors declare no competing interests.

\section*{Acknowledgements}
This work is supported by the National Natural Science Foundation of China (No. 12201322 to B.Y., No. 12401530 to G.Y., No. 12461080 to Y.L and No. 12561068 to H.L.), 
Scientific Research Project of Higher Education Institutions of Inner Mongolia Autonomous Region (No. NJZY23003 to Z. D.), 
Natural Science Foundation of Inner Mongolia (No. 2025MS01003 to B.Y.),
Program for Innovative Research Team in Universities of Inner Mongolia Autonomous Region (No. NMGIRT2413 to Y.L.),
Key Project of Natural Science Foundation of Inner Mongolia Autonomous Region (No. 2025ZD036 to H.L.).

\bibliography{mybibfile}

\end{document}